\documentclass[article,12pt]{amsart}

\usepackage{amsfonts}
\usepackage{amsmath}
\usepackage{mathrsfs}
\usepackage{graphicx}
\usepackage{color}
\usepackage{epsfig}
\usepackage{yfonts}
\usepackage{fancybox}
\usepackage{enumerate}
\usepackage{dsfont}
\usepackage{etoolbox}
\apptocmd{\thebibliography}{}{}{}

\numberwithin{equation}{section}
\theoremstyle{plain}
\newtheorem{thm}{Theorem}[section]
\newtheorem{rem}{Remark}[section]

\newtheorem{lem}{Lemma}[section]
\newtheorem{deff}{Definition}[section]

\newcommand{\dE}{\mathbb{E}}
\newcommand{\dR}{\mathbb{R}}
\newcommand{\dL}{\mathbb{L}}

\newcommand{\dC}{\mathbb{C}}

\newcommand{\cN}{\mathcal{N}}

\newcommand{\rI}{\mathrm{I}}
\newcommand{\cF}{\mathcal{F}}

\newcommand{\cR}{\mathcal{R}}

\newcommand{\ind}{\mbox{1}\kern-.25em \mbox{I}}
\font\calcal=cmsy10 scaled\magstep1
\def\build#1_#2^#3{\mathrel{\mathop{\kern 0pt#1}\limits_{#2}^{#3}}}
\def\liml{\build{\longrightarrow}_{}^{{\mbox{\calcal L}}}}

\def\videbox{\mathbin{\vbox{\hrule\hbox{\vrule height1.4ex \kern.6em\vrule height1.4ex}\hrule}}}
\def\demend{\hfill $\videbox$\\}
\email{bernard.bercu@math.u-bordeaux.fr}

\keywords{Elephant random walk; Martingales; Strong law of large numbers;
Asymptotic normality}
\subjclass[2010]{Primary:  60G50; Secondary: 60G42; 60F05}

\begin{document}
\title[On the elephant random walk with stops]
{On the elephant random walk with stops playing hide and seek with the Mittag-Leffler distribution\vspace{1ex}}
\author{Bernard Bercu}
\dedicatory{\normalsize University of Bordeaux, France}
\address{Universit\'e de Bordeaux, Institut de Math\'ematiques de Bordeaux,
UMR CNRS 5251, 351 Cours de la Lib\'eration, 33405 Talence cedex, France.}
\thanks{}

\begin{abstract}
The aim of this paper is to investigate the asymptotic
behavior of the so-called elephant random walk with stops (ERWS). 
In contrast with the standard elephant random walk, the elephant is allowed to be lazy 
by staying on his own position. We prove that the number of ones of the ERWS, properly
normalized, converges almost surely to a Mittag-Leffler distribution.
It allows us to carry out a sharp analysis on the asymptotic behavior of the ERWS.
In the diffusive and critical regimes, we establish the almost sure convergence
of the ERWS. We also show that it is necessary to self-normalized the position
of the ERWS by the random number of ones in order to prove the asymptotic normality.
In the superdiffusive regime, we establish the almost sure convergence of
the ERWS, properly normalized, to a nondegenerate random variable.
Moreover, we also show that the fluctuation of the ERWS around its limiting random variable is 
still Gaussian.
\end{abstract}
\maketitle


\ \vspace{-5ex}
\section{Introduction}
\label{S-I}

Discrete-time random walks with long-memory arose naturally in mathematics and statistical physics. One of them is the famous elephant random walk (ERW) which was introduced in the early 2000s by 
Sch\"utz and Trimper \cite{Schutz2004} in order to investigate how long-range memory affects its asymptotic behavior. The ERW shows three different regimes depending on the location of its memory parameter $p\in [0,1]$ and is defined as follows.
The elephant starts at the origin at time zero, $X_0=0$. For the first step, $X_1$ 
has a Rademacher $\cR(s)$ distribution which means that the elephant moves to 
the right at point $1$ with probability $s$ or to the left at point $-1$ with probability $1-s$ 
where $s$ lies in $[0,1]$. Then, the elephant chooses uniformly at random 
an integer $k$ among the previous times $1,\ldots,n$ and it moves exactly in the same direction as that of time $k$ with probability $p$ 
or in the opposite direction with probability $1-p$. In other words,
\begin{equation}
\label{STEPSERW}
   X_{n+1} = \left \{ \begin{array}{ccc}
    +X_{k} &\text{ with probability } & p, \vspace{2ex}\\
    -X_{k} &\text{ with probability } & 1-p.
   \end{array} \right.
\end{equation}
Over the last decade, the ERW has received a growing attention in mathematics and statistical physics in the diffusive regime $p< 3/4$ and the critical regime $p=3/4$, see \cite{Baur2016, Coletti2017, Colettib2017}, as well as in the superdiffusive regime $p>3/4$, see \cite{Bercu2018, Kubota2019}. We also refer the reader to 
the extension to the multi-dimensional ERW \cite{BercuLaulin2019, Bercu2020, Bertenghi2020, Gonzalez2020}
as well as to the recent contributions \cite{Bertoinb2021, Bertoin2021, Businger2018, Coletti2019, Fan2020, Takei2020, Vazquez2019}. 
\ \vspace{2ex}\\
Surprisingly, to the best of our knowledge, very few results are available on the elephant random walk with stops (ERWS) except Cressoni et al. \cite{Cressoni2013} dealing with the calculation of 
the moments of the ERWS, Gut and Stadtm\"uller \cite{Gut2021} concerning the asymptotic behavior of the number of zeros of the ERWS, and Bercu et al. \cite{Bercu2019} for new hypergeometric 
identities arising from the ERWS. The ERWS is defined
as follows. As previously seen, the elephant starts at the origin at time zero $X_0=0$
and $X_1$ has a Rademacher $\cR(s)$ distribution where $s$ lies in $[0,1]$.
Then, the elephant chooses uniformly at random 
an integer $k$ among the previous times $1,\ldots,n$ and $X_{n+1}$ is determined
stochastically by
\begin{equation}
\label{STEPSERWS}
   X_{n+1} = \left \{ \begin{array}{ccc}
    +X_{k} &\text{ with probability } & p, \vspace{1ex}\\
    -X_{k} &\text{ with probability } & q, \vspace{1ex}\\
    0 &\text{ with probability } & r,
   \end{array} \right.
\end{equation}
where $p+q+r=1$. Throughout the paper, we assume that $0<r<1$ inasmuch as 
the case $r=0$ corresponds to the standard ERW which was previously investigated while in the case
$r=1$, the ERWS remains stuck at zero after the first step.
The position of the ERWS is given by $S_0=0$ and for all $n \geq 1$,
\begin{equation}
\label{POSERWS}
S_{n}=\sum_{k=1}^nX_{k}.
\end{equation}
The motivation for studying the ERWS is twofold. On the one hand, we shall show how the ERWS
plays hide and seek with the Mittag-Leffler distribution. On the other hand, we shall make use of non-standard results on martingales 
in order to overcome the fact that the random number of ones of 
the ERWS has a prominent role in the analysis of its asymptotic behavior, which has not been taken 
into account in \cite{Gonzalez2021, Gut2019}.
More precisely, denote by
$\Sigma_n$ the number of ones of the ERWS up to time $n$,
\vspace{-1ex}
\begin{equation}
\label{ONES}
\Sigma_{n}=
\sum_{k=1}^n X_{k}^2.
\end{equation}
We shall improve Theorem 3.1 in \cite{Gut2021} by showing that, whatever the values of the parameters $p$, $q$ and $r$
in $[0,1]$,
\begin{equation}
\label{ASCVG1ML}
\lim_{n \rightarrow \infty} \frac{1}{n^{1-r}}\Sigma_n = \frac{1}{\Gamma(2-r)}\Sigma \hspace{1cm}\text{a.s.}
\end{equation}
where $\Sigma$ stands for a Mittag-Leffler distribution with parameter $1-r$ and $\Gamma$ is the
Euler Gamma function. One can observe that \eqref{ASCVG1ML} also holds for the ERW, as in the case $r=0$, we clearly have $\Sigma_n=n$ and the  Mittag-Leffler distribution with parameter  $1$ reduces 
to $\Sigma=1$.
The almost sure convergence \eqref{ASCVG1ML} will allow us to carry out
a sharp analysis on the asymptotic behavior of the ERWS. 
\newpage
\noindent
We shall see that the ERWS shows three different regimes depending on the location 
of the memory parameter
\begin{equation}
\label{MEMPAR}
p_r=\frac{p}{1-r}.
\end{equation}
The ERWS is said to be diffusive if $p_r<3/4$, critical if $p_r = 3/4$ and
superdiffusive if $3/4 < p_r$.
The paper is organized as follows. Section \ref{S-ML} deals with the almost sure convergence to the Mittag-Leffler distribution while
Section \ref{S-MR} is devoted to the main results of the paper. We establish the almost sure asymptotic behavior of the
ERWS in the diffusive, critical and superdiffusive regimes. Moreover, we also prove the asymptotic normality of the ERWS, suitably normalized by
$\Sigma_n$, in the diffusive and critical regimes. Finally, the fluctuation of the ERWS around its limiting random variable is also provided
in the superdiffusive regime. 
Two keystone martingales are analyzed in Section \ref{S-MA}. All technical proofs are
postponed to Appendices A to C.

\section{On the Mittag-Leffler distribution}
\label{S-ML}

The Mittag-Leffler function was introduced at the beginning of the last century and was widely studied by various mathematicians.
It is defined, for all $z \in \dC$, by
$$
E_\alpha(z) = \sum_{n=0}^\infty \frac{z^n}{\Gamma(1+n\alpha )}
$$
where $\alpha$ is a positive real parameter and $\Gamma$ stands for the Euler Gamma function. One can observe that $E_1(z)=\exp(z)$ while
$E_2(-z^2)=\cos(z)$.

\begin{deff}
We shall say that a positive random variable $X$ has a Mittag-Leffler distribution with parameter $\alpha \in [0,1]$ if its Laplace transform is given, 
for all $t\in \dR$, by
$$
\dE[\exp(tX)]=E_\alpha(t) = \sum_{n=0}^\infty \frac{t^n}{\Gamma(1+n\alpha )}.
$$
Consequently, for any integer $m \geq 1$, 
\begin{equation}
\label{DEFMLMOM}
\dE[X^m]=\frac{m!}{\Gamma(1+m \alpha )}.
\end{equation}
\end{deff}
\noindent
The Mittag-Leffler distribution satisfies the famous Carleman's condition which means that it is 
characterized by its moments. 
If $\alpha=0$, then $X$ has an exponential distribution with parameter $1$, while if $\alpha=1$, 
$X$ is concentrated on the value $1$.
If $0<\alpha<1$, the probability density function of $X$ was explicitely calculated by 
Pollard \cite{Pollard1948}, see also Feller \cite{Feller1971}. It is given by
$$
f_\alpha(x)=\frac{1}{\pi \alpha} \sum_{n=0}^\infty \Gamma(1+\alpha n)\sin(\alpha n \pi) \frac{(-x)^{n-1}}{n!}\rI_{\{x>0\}}.
$$
As a special case,
$$
f_{1/2}(x)=\frac{1}{\sqrt{\pi}} \exp\Big(\! -\frac{x^2}{4}\Big)\rI_{\{x>0\}}.
$$ 
It means that the Mittag-Leffler distribution with parameter $\alpha=1/2$ coincides with the distribution of 
$|Z|$ where $Z$ has a Gaussian $\cN(0,2)$ distribution. We also refer the reader to Janson \cite{Janson2006} where 
the Mittag-Leffler distribution appears as the asymptotic distribution 
of the composition of a generalized P\'olya urn with two colors. We are now in position 
to formulate our first result.

\begin{lem}
\label{L-ML}
Whatever the values of $p$, $q$ and $r$ in $[0,1]$, we have the almost sure convergence
\begin{equation}
\label{ASCVGML}
\lim_{n \rightarrow \infty} \frac{1}{n^{1-r}}\Sigma_n = 
\Sigma \hspace{1cm}\text{a.s.}
\end{equation}
where $\Sigma$ has a Mittag-Leffler distribution with parameter $1-r$. 
Consequently, $\Sigma$ is positive with probability one.
Moreover, this convergence holds in $\dL^m$ for any integer $m\geq 1$. Hence, for any integer 
$m\geq 1$,
\begin{equation}
\label{MLMOM}
\dE[\Sigma^m]=\frac{m!}{\Gamma(1+m(1-r))}.
\end{equation}
\end{lem}

\begin{proof}
The proof is given in Section \ref{S-MA}. 
\end{proof}

\begin{rem}
On the one hand, the almost sure convergence \eqref{ASCVGML} was already established in 
Theorem 3.1 of \cite{Gut2021} as well as the convergence in $\dL^1$. One can notice a slight difference in the
value of $\dE[\Sigma]$ due to the different first step $X_1$. Moreover, the distribution of $\Sigma$ has not been identified in \cite{Gut2021}. 
A crucial point in all the sequel is that $\Sigma$ is almost surely positive. On the other hand, our proof is totally different from that of Theorem 1  in \cite{Takei2020}. More precisely, it relies on the calculation of the Pochhammer moments of the random variable 
$\Sigma_n$ together with a nice identity that links the Pochhammer and the classical moments of
$\Sigma_n$ via the unsigned Stirling numbers of the first kind.
\end{rem}


\ \vspace{-2ex}
\section{Main results}
\label{S-MR}


\subsection{The diffusive regime}

Our first result deals with the almost sure convergence of the ERWS
in the diffusive regime where $p_r<3/4$.

\begin{thm}
\label{T-ASCVG-DR}
We have the almost sure convergence
\begin{equation}
\label{ASCVGDR}
\lim_{n \rightarrow \infty} \frac{S_n}{n} = 0 \hspace{1cm}\text{a.s.}
\end{equation}
\end{thm}

\noindent
Some refinements given by the law of iterated logarithm are as follows. One can observe that it is necessary
to self-normalized the position $S_n$ by the numbers of ones $\Sigma_n$ of the ERWS in order 
to prove the  law of iterated logarithm. Denote by $\sigma_r^2$ the asymptotic variance
given by
\begin{equation}
\label{VARDR}
\sigma_r^2=\frac{1-r}{3(1-r)-4p}.
\end{equation}

\begin{thm}
\label{T-LIL-DR}
We have the the law of the iterated logarithm
\begin{equation}
 \limsup_{n \rightarrow \infty} \frac{S_n}{\sqrt{2 \Sigma_n \log \log \Sigma_n}}   = 
 -\liminf_{n \rightarrow \infty} \frac{S_n}{\sqrt{2 \Sigma_n \log \log \Sigma_n}} 
 = \sigma_r\hspace{1cm} \text{a.s.}
 \label{LIL-DR1}
\end{equation}
leading to
\begin{equation}
 \limsup_{n \rightarrow \infty} \frac{S_n^2}{2 \Sigma_n \log \log \Sigma_n}=  
\sigma_r^2 \hspace{1cm} \text{a.s.}
 \label{LIL-DR2}
\end{equation}
Moreover, we also have 
\begin{equation}
 \limsup_{n \rightarrow \infty} \frac{S_n}{\sqrt{2 n^{1-r} \log \log n}}   = 
 -\liminf_{n \rightarrow \infty} \frac{S_n}{\sqrt{2 n^{1-r} \log \log n}} 
 = \sigma_r \sqrt{\Sigma}\hspace{1cm} \text{a.s.}
 \label{LIL-DR3}
\end{equation}
implying that
\begin{equation}
 \limsup_{n \rightarrow \infty} \frac{S_n^2}{2 n^{1-r} \log \log n}=  
 \sigma_r^2 \Sigma \hspace{1cm} \text{a.s.}
 \label{LIL-DR4}
\end{equation}
where $\Sigma$ is the limiting random variable given in \eqref{ASCVGML}.
\end{thm}

\begin{rem}
We immediately deduce from \eqref{LIL-DR4} the almost sure rate of convergence
\begin{equation}
\label{RATEDR}
\Big(\frac{S_n}{n}\Big)^2 = O\Big( \frac{\log \log n}{n^{1+r}} \Big) \hspace{1cm}\text{a.s.}
\end{equation}
\end{rem}

\noindent
Our next result concerns the asymptotic normality of the ERWS. As previously seen, it is necessary
to self-normalized the position $S_n$ by the numbers of ones $\Sigma_n$ of the ERWS in order 
to establish the asymptotic normality. 

\begin{thm}
\label{T-AN-DR}
We have the asymptotic normality
\begin{equation}
\label{ANDR}
\frac{S_n}{\sqrt{\Sigma_n}} \liml \cN\big(0, \sigma_r^2\big).
\end{equation}
Moreover, we also have
\begin{equation}
\label{ANDRML}
\frac{S_n}{\sqrt{n^{1-r}}} \liml \sqrt{\Sigma^\prime}\cN\big(0, \sigma_r^2\big)
\end{equation}
where $\Sigma^\prime$ is independent of the Gaussian $\cN\big(0, \sigma_r^2\big)$ random variable
and $\Sigma^\prime$ has a Mittag-Leffler distribution with parameter $1-r$.
\end{thm}

\begin{rem}
In the special case $r=0$, the elephant random walk without stop coincides
with the standard ERW.  Consequently, $\Sigma_n=n$, $\Sigma=1$ and $\sigma_r^2$ reduces to
$$
\sigma^2=\frac{1}{3-4p}.
$$
Hence, we find again the law of iterated logarithm given by Theorem 3.2 in \cite{Bercu2018}
as well as the asymptotic normality given by Theorem 2 in \cite{Coletti2017} or Theorem 3.3 in \cite{Bercu2018}.
\end{rem}


\subsection{The critical regime}

Hereafter, we focus our attention on the almost sure convergence of the ERWS in the critical
regime where $p_r=3/4$.

\begin{thm}
\label{T-ASCVG-CR}
We have the almost sure convergence
\begin{equation}
\label{ASCVGCR}
\lim_{n \rightarrow \infty} \frac{S_n}{n} = 0 \hspace{1cm}\text{a.s.}
\end{equation}
\end{thm}

\noindent
Our next result deals with to the law of iterated logarithm of the ERWS with a quite unusual rate of convergence.

\begin{thm}
\label{T-LIL-CR}
We have the law of the iterated logarithm
\begin{eqnarray}
 \limsup_{n \rightarrow \infty} \frac{S_n}{\sqrt{2 \Sigma_n \log \Sigma_n \log \log \log \Sigma_n}}  
&  = &
 -\liminf_{n \rightarrow \infty} \frac{S_n}{\sqrt{2 \Sigma_n \log \Sigma_n \log \log \log \Sigma_n}} 
 \notag \\
 & =& 1\hspace{1cm} \text{a.s.}
 \label{LIL-CR1}
\end{eqnarray}
leading to
\begin{equation}
 \limsup_{n \rightarrow \infty} \frac{S_n^2}{2 \Sigma_n \log \Sigma_n \log \log \log \Sigma_n}=  
1\hspace{1cm} \text{a.s.}
 \label{LIL-CR2}
\end{equation}
Moreover, we also have 
\begin{eqnarray}
 \limsup_{n \rightarrow \infty} \frac{S_n}{\sqrt{2 n^{1-r} \log n \log \log \log n}}  
&  = &
 -\liminf_{n \rightarrow \infty} \frac{S_n}{\sqrt{2 n^{1-r} \log n \log \log \log n}} 
 \notag \\
 & =& \sqrt{(1-r) \Sigma}\hspace{1cm} \text{a.s.}
 \label{LIL-CR3}
\end{eqnarray}
implying that
\begin{equation}
 \limsup_{n \rightarrow \infty} \frac{S_n^2}{2 n^{1-r} \log n \log \log \log n}=  
 (1-r) \Sigma \hspace{1cm} \text{a.s.}
 \label{LIL-CR4}
\end{equation}
where $\Sigma$ is the limiting random variable given in \eqref{ASCVGML}.
\end{thm}

\begin{rem}
We clearly obtain from \eqref{LIL-CR4} the almost sure rate of convergence
\begin{equation}
\label{RATECR}
\Big(\frac{S_n}{n}\Big)^2 = O\Big( \frac{\log n \log \log \log n}{n^{1+r}} \Big) \hspace{1cm}\text{a.s.}
\end{equation}
\end{rem}

\noindent
The asymptotic normality of the ERWS is as follows.

\begin{thm}
\label{T-AN-CR}
We have the asymptotic normality
\begin{equation}
\label{ANCR}
\frac{S_n}{\sqrt{\Sigma_n \log \Sigma_n }} \liml \cN(0, 1).
\end{equation}
Moreover, we also have
\begin{equation}
\label{ANCRML}
\frac{S_n}{\sqrt{n^{1-r} \log n}} \liml \sqrt{(1-r)\Sigma^\prime}\cN(0,1)
\end{equation}
where $\Sigma^\prime$ is independent of the Gaussian $\cN(0, 1)$ random variable
and $\Sigma^\prime$ has a Mittag-Leffler distribution with parameter $1-r$.
\end{thm}

\begin{rem}
As previously seen, in the special case of the standard ERW where $r=0$, we have $\Sigma_n=n$.
Hence, we find again the law of iterated logarithm given by Theorem 3.5 in \cite{Bercu2018}
as well as the asymptotic normality given by Theorem 2 in \cite{Coletti2017}
or Theorem 3.6 in \cite{Bercu2018}.
\end{rem}


\subsection{The superdiffusive regime}

We next investigate the almost sure convergence of the ERWS in the superdiffusive
regime where $p_r>3/4$.

\begin{thm}
\label{T-ASCVG-SR}
We have the almost sure convergence
\begin{equation}
\label{ASCVGSR1}
\lim_{n \rightarrow \infty} \frac{S_n}{n^{2p+r-1}} = L \hspace{1cm}\text{a.s.}
\end{equation}
where $L$ is a non-degenerate random variable. 
Moreover, this convergence holds in $\dL^m$ for any integer $m\geq 1$,
which means that for any integer $m\geq 1$,
\begin{equation}
\label{ASCVGSR2}
 \lim_{n \rightarrow \infty} \dE\Bigl[ \Bigl| \frac{S_n}{n^{2p+r-1}} -L \Bigr|^m \Bigr]=0.
\end{equation}
\end{thm}

\begin{thm}
\label{T-MOM-SR}
The first four moments of $L$ are given by 
\begin{eqnarray}
\label{MOML1}
\dE[L]  &=& \frac{2s-1}{(2p+r-1)\Gamma(2p+r-1)},\\
\label{MOML2}
\dE[L^2] &=& \frac{1}{(4p+3(r-1))\Gamma(2(2p+r-1))}, \\
\label{MOML3}
\dE[L^3] &=& \frac{2p(2s-1)}{(2p+r-1)(4p+3(r-1))\Gamma(3(2p+r-1))}, \\
\label{MOML4}
\dE[L^4] &=& \frac{6\big(8p^2-4p(1-r)-(1-r)^2\big)}{(8p+5(r-1))(4p+3(r-1))^2\Gamma(4(2p+r-1))}.
\end{eqnarray}
\end{thm}

\begin{rem}
One can observe that in the special case of the standard ERW where $r=0$, we find again
the first four moments of $L$ given by Theorem 3.8 in \cite{Bercu2018}.
\end{rem}

\noindent
Our last result concerns the fluctuation of the ERWS around its limiting random variable $L$, 
in the spirit of the original work of Kubota and Takei \cite{Kubota2019}. It shows that the fluctuation 
of the position $S_n$ around $n^{2p+r-1}L$ is still Gaussian. As previously seen,
one can observe that $S_n$ is self-normalized by the numbers of ones $\Sigma_n$ of the ERWS.
Denote by $\tau_r^2$ the asymptotic variance, quite similar to  $\sigma_r^2$ except
for the sign in the denominator, given by
\begin{equation}
\label{VARSR}
\tau_r^2=\frac{1-r}{4p-3(1-r)}.
\end{equation}

\begin{thm}
\label{T-AN-SR}
We have the asymptotic normality
\begin{equation}
\label{ANSR1}
\frac{S_n-n^{2p+r-1}L}{\sqrt{\Sigma_n}}
\liml \cN\big(0,\tau_r^2\big).
\end{equation}
Moreover, we also have
\begin{equation}
\label{ANSR2}
\frac{S_n -n^{2p+r-1}L }{\sqrt{n^{1-r}}} 
\liml \sqrt{\Sigma^\prime}\cN\big(0,\tau_r^2\big)
\end{equation}
where $\Sigma^\prime$ is independent of the Gaussian $\cN\big(0,\tau_r^2\big)$ random variable
and $\Sigma^\prime$ has a Mittag-Leffler distribution with parameter $1-r$.
\end{thm}

\begin{rem}
In the special case  where $r=0$, we have $\Sigma_n=n$ and
we find again a simplified version of Theorem 2.3 in \cite{Kubota2019}.
\end{rem}

\vspace{-2ex}
\section{Two keystone martingales}
\label{S-MA}

\noindent
Denote by $(\cF_n)$ the increasing sequence of $\sigma$-algebras generated by the ERWS, $\mathcal{F}_n=\sigma\left(X_1,\ldots,X_n\right)$.
It follows from \eqref{STEPSERWS} that the steps satisfy, for all $n\geq 1$,
\begin{equation}
\label{CM12X}
\dE[X_{n+1}|\cF_n]=a\frac{S_n}{n}
\hspace{1cm}\text{and}\hspace{1cm}
\dE[X_{n+1}^2|\cF_n]=b\frac{\Sigma_n}{n}
\end{equation}
where the two fundamental parameters $a$ and $b$ are given by
\begin{equation}
\label{DEFAB}
a=p-q
\hspace{1cm}\text{and}\hspace{1cm}
b=p+q.
\end{equation}
Consequently, we clearly obtain from \eqref{POSERWS}, \eqref{ONES} and \eqref{CM12X} that almost surely
\begin{eqnarray}
\label{CM1SUM}
\dE[S_{n+1}|\cF_n] &=& \dE[S_n+ X_{n+1}|\cF_n]= S_n + \frac{a}{n} S_n= \alpha_n S_n, \\
\dE[\Sigma_{n+1}|\cF_n] &=& \dE[\Sigma_n+X_{n+1}^2|\cF_n] = \Sigma_n + \frac{b}{n} \Sigma_n= \beta_n \Sigma_n,
\label{CM1SIG}
\end{eqnarray}
where
\begin{equation*}
\alpha_n=1+\frac{a}{n}
\hspace{1cm}\text{and}\hspace{1cm}
\beta_n=1+\frac{b}{n}.
\end{equation*}
Let $(a_n)$ and $(b_n)$ be the two sequences defined by $a_1=1$, $b_1=1$ and for $n\geq 2$,
\begin{equation}
\label{DEFABN}
a_n=\prod_{k=1}^{n-1}\alpha_k^{-1}=\frac{\Gamma(n)\Gamma(a+1)}{\Gamma(n+a)}
\hspace{1cm}\text{and}\hspace{1cm}
b_n=\prod_{k=1}^{n-1}\beta_k^{-1}=\frac{\Gamma(n)\Gamma(b+1)}{\Gamma(n+b)}.
\end{equation}
Let $(M_n)$ and $(N_n)$ be the two sequences
defined, for all $n \geq 1$, by 
\begin{equation}
\label{DEFMN}
M_n=a_nS_n
\hspace{1cm}\text{and}\hspace{1cm}
N_n=b_n\Sigma_n.
\end{equation}
Since $a_n= \alpha_n a_{n+1}$ and
$b_n=\beta_n b_{n+1}$, we immediately deduce from \eqref{CM1SUM} and \eqref{CM1SIG} that 
$\dE[M_{n+1}|\cF_n]=M_n$ and $\dE[N_{n+1}|\cF_n]=N_n$ a.s.
It means that $(M_n)$ and $(N_n)$ are two discrete-time martingale sequences. 
We shall see below that the asymptotic behavior of $(M_n)$ is totally different from that of $(N_n)$.
The martingales $(M_n)$ and $(N_n)$ can be rewritten in the additive form
\begin{equation} 
\label{NEWDEFMN}
M_n=\sum_{k=1}^n a_k \varepsilon_k
\hspace{1cm}\text{and}\hspace{1cm}
N_n=\sum_{k=1}^n b_k \xi_k
\end{equation}
where the martingale increments $\varepsilon_n=S_{n}-\alpha_{n-1} S_{n-1}$
and $\xi_n=\Sigma_{n}-\beta_{n-1} \Sigma_{n-1}$.
The predictable quadratic variation associated with $(M_n)$ is given by
$\langle M \rangle_0=0$ and, for all $n \geq 1$,
\vspace{-1ex}
\begin{equation} 
\label{IPM}
\langle M \rangle_n = \sum_{k=1}^n a_k^2\dE[\varepsilon_k^2| \mathcal{F}_{k-1}].
\end{equation}
It follows from \eqref{POSERWS} together with \eqref{CM12X} that
\begin{eqnarray}
\dE[S_{n+1}^2|\cF_n]&=&\dE[S_{n}^2+2S_n X_{n+1} +X_{n+1}^2|\cF_n]=S_n^2+\frac{2a}{n}S_n^2 +b\frac{\Sigma_n}{n}  \hspace{1cm} \text{a.s.} 
\notag\\
&=&b\frac{\Sigma_n}{n} +(2\alpha_n -1) S_n^2  \hspace{1cm} \text{a.s.}
\label{CM2SUM}
\end{eqnarray}
Hence, as $\dE[\varepsilon^2_{n+1}|\cF_n]= \dE[S_{n+1}^2|\cF_{n}] - \alpha_{n}^2 S_{n}^2$, 
we obtain from \eqref{CM2SUM} that 
\begin{equation} 
\dE[\varepsilon^2_{n+1}|\cF_n]
=b\frac{\Sigma_n}{n} -(1-\alpha_n)^2 S_n^2=b\frac{\Sigma_n}{n} - \Big(\frac{aS_n}{n}\Big)^2 \hspace{1cm} \text{a.s.} 
\label{CM2EPS}
\end{equation}
which ensures that
\begin{equation}
\label{SUPEPS2}
\sup_{n\geq 0} \dE[\varepsilon^2_{n+1}|\cF_n]\leq b \hspace{1cm} \text{a.s.}
\end{equation}
Therefore, we deduce from \eqref{IPM} and \eqref{CM2EPS} that $\langle M \rangle_n$ can be splitted into two terms

\begin{equation}
\label{CALCIPM}
\langle M \rangle_n = 1+ b V_n - a^2 W_n
\end{equation}
where
$$
V_n = \sum_{k=1}^{n-1} a_{k+1}^2 \Big(\frac{\Sigma_k}{k}\Big)
\hspace{1cm}\text{and}\hspace{1cm}
W_n = \sum_{k=1}^{n-1} a_{k+1}^2 \Big(\frac{S_k}{k}\Big)^2.
$$
The main difficulty arising here is that the sequence $(V_n)$, even properly normalized, will not converge to a deterministic value but to
a random variable. It has not been taken into account in \cite{Gonzalez2021, Gut2019}, leading to erroneous results.
Hereafter, we shall focus our attention on the asymptotic behavior of the sequence $(\Sigma_n)$, 
in order to establish the asymptotic behavior of the second martingale $(N_n)$.

\begin{lem}
\label{L-MARTN}
The martingale $(N_n)$ is bounded in $\dL^m$ for any integer $m\geq 1$. 
More precisely, for all $n\geq 1$ and for any integer $m\geq 1$, 
\begin{equation}
\label{MOMN}
\dE[N_n^m] \leq m!.
\end{equation}
Consequently, $(N_n)$ converges almost surely and in $\dL^m$ to a finite random variable $N$ satisfying
for any integer $m\geq 1$, 
\begin{equation}
\label{CALCMOMN}
\dE[N^m] =\frac{m!(\Gamma(b+1))^m}{\Gamma(1+mb)}.
\end{equation}
\end{lem}

\begin{proof}
In order to prove \eqref{MOMN}, we are going to compute the Pochhammer moments of the random variable $\Sigma_n$, also known as the rising factorial moments of $\Sigma_n$, defined by
$\dE[\Sigma_n^{(1)}]=\dE[\Sigma_n]$ and for any integer $m\geq 1$,
$$
\dE[\Sigma_n^{(m)}]=\dE[\Sigma_n(\Sigma_n+1)\cdots(\Sigma_n+m-1)].
$$
One can observe that $\Sigma_n$ is always smaller than $n$ which means that 
$\dE[\Sigma_n^{(m)}] \leq n^{(m)}$ where, for any $a \in \dR$, $a^{(m)}=a(a+1)\cdots(a+m-1)$ stands for the Pochhammer symbol of $a$
with $a^{(0)}=1$.
We recall that Pochhammer symbols are Sheffer sequences of binomial type satisfying, for any $a,b \in \dR$, the nice identity
$$
(a+b)^{(m)}=\sum_{k=0}^m
\begin{pmatrix}
m \\ k
\end{pmatrix}
a^{(k)}b^{(m-k)}.
$$
Consequently, for any integer $m \geq 1$,
\begin{equation}
\label{NP1}
\Sigma_{n+1}^{(m)}=(\Sigma_{n}+X_{n+1}^2)^{(m)}=\sum_{k=0}^m
\begin{pmatrix}
m \\ k
\end{pmatrix}
\Sigma_n^{(k)}(X_{n+1}^2)^{(m-k)}.
\end{equation}
By taking the conditional expectation on both sides of \eqref{NP1}, we obtain that
\begin{equation}
\label{NP2}
\dE[\Sigma_{n+1}^{(m)}| \cF_n]=\sum_{k=0}^m
\begin{pmatrix}
m \\ k
\end{pmatrix}
\Sigma_n^{(k)}\dE[ (X_{n+1}^2)^{(m-k)}| \cF_n].
\end{equation}
However, it follows from \eqref{STEPSERWS} that for all $k\geq 1$,
\begin{eqnarray}
\dE[ (X_{n+1}^2)^{(k)}| \cF_n]&=&\dE[X_{n+1}^2(X_{n+1}^2+1) \cdots (X_{n+1}^2+k-1)|\cF_n], \notag\\
&=&1^{(k)}\frac{b}{n}\Sigma_n=k!\frac{b}{n}\Sigma_n \hspace{1cm} \text{a.s.}
\label{NP3}
\end{eqnarray}
Therefore, we deduce from \eqref{NP2} and \eqref{NP3} that for any integer $m \geq 1$,
\begin{eqnarray*}
\dE[\Sigma_{n+1}^{(m)}| \cF_n]&=&\Sigma_n^{(m)}+\sum_{k=0}^{m-1}
\begin{pmatrix}
m \\ k
\end{pmatrix}
\Sigma_n^{(k)} (m-k)!\frac{b}{n}\Sigma_n \hspace{1cm} \text{a.s.}  \\
&=&\Sigma_n^{(m)}+m! \frac{b}{n} \sum_{k=0}^{m-1}
\frac{1}{k!} \Sigma_n
\Sigma_n^{(k)} \hspace{1cm} \text{a.s.} 
\end{eqnarray*}
leading to
\begin{equation}
\label{NP4}
\dE[\Sigma_{n+1}^{(m)}]=\dE[\Sigma_n^{(m)}]+m! \frac{b}{n} \sum_{k=0}^{m-1}
\frac{1}{k!} \dE[\Sigma_n\Sigma_n^{(k)}].
\end{equation}
From now on, the Pochhammer moments will play a prominant role.
It is easy to see that for all $k\geq 0$ and for any $a \in \dR$, $a a^{(k)}=a^{(k+1)}-ka^{(k)}$.
It clearly implies that $ \dE[\Sigma_n\Sigma_n^{(k)}]=  \dE[\Sigma_n^{(k+1)}]  -k\dE[\Sigma_n^{(k)}]$.
Hence, we obtain from \eqref{NP4} together with a standard telescoping argument that
\begin{eqnarray}
\dE[\Sigma_{n+1}^{(m)}]&=&\dE[\Sigma_n^{(m)}]+m! \frac{b}{n}\Bigl( 
\dE[\Sigma_n]+\sum_{k=1}^{m-1}
\frac{1}{k!} \dE[\Sigma_n^{(k+1)}] - \frac{1}{(k-1)!} \dE[\Sigma_n^{(k)}]\Bigr), \notag \\
&=& \dE[\Sigma_n^{(m)}]+m! \frac{b}{n}\Bigl( 
\dE[\Sigma_n] + \frac{1}{(m-1)!} \dE[\Sigma_n^{(m)}] - \dE[\Sigma_n] \Bigr),\notag \\
&=& \Bigl(1+ \frac{mb}{n}\Bigr) \dE[\Sigma_n^{(m)}]. 
\label{NP5}
\end{eqnarray}
Furthermore, one can observe that $\Sigma_1=X_1^2=1$, which means that $\dE[ \Sigma_1^{(m)}]=m!$. Consequently, equation \eqref{NP5} leads, for all $n\geq 2$ and for any integer $m \geq 1$, to
\begin{equation}
\label{NP6}
\dE[\Sigma_{n}^{(m)}]=\prod_{k=1}^{n-1}  \Bigl(1+ \frac{mb}{k}\Bigr) \dE[ \Sigma_1^{(m)}]= \frac{m! \Gamma(n+mb)}{\Gamma(n) \Gamma(1+mb)}.
\end{equation}
For all $m \geq 1$, denote
\begin{equation}
\label{DEFBM}
b_n(m)=\frac{\Gamma(n+mb)}{\Gamma(n) \Gamma(1+mb)}.
\end{equation}
One can obviously see that $b_n(1)b_n=1$ where the sequence $(b_n)$ was previously defined in \eqref{DEFABN}. Moreover,
it is easy to see from \eqref{DEFABN} that for all $n \geq 2$ and $m \geq 1$,
\begin{equation}
\label{NP7}
b_n(m)  b_n^m= \prod_{k=1}^{n-1}  \Bigl(\frac{k+mb}{k}\Bigr) \Bigl(\frac{k}{k+b}\Bigr)^m \leq 1. 
\end{equation}
%
%
Since $N_n=b_n \Sigma_n$, we deduce from
\eqref{NP6} and \eqref{NP7} that for all $n \geq 2$ and $m \geq 1$,
\begin{equation}
\label{NP8}
\dE[N_n^m]=b_n^m  \dE[\Sigma_{n}^m] \leq b_n^m \dE[\Sigma_{n}^{(m)}]= b_n(m)  b_n^m m! \leq m!
\end{equation}
which is exactly inequality \eqref{MOMN}. Therefore, we immediately obtain from \eqref{NP8} that
the martingale $(N_n)$ is bounded in $\dL^m$ for any integer $m\geq 1$. Then, it follows from Doob's martingale convergence
Theorem, see e.g. Corollary 2.2 in \cite{Hall1980}, that $(N_n)$ converges almost surely and in $\dL^m$ 
to a finite random variable $N$. It only remains to calculate all the moments of the limiting random variable $N$.
For any $a \in \dR$ and for any integer $m\geq 1$, we have the very nice identity
\begin{equation}
\label{STIRLING}
a^{(m)}=\sum_{k=0}^m 
\Bigl[ 
\begin{array}{c}
m \\
k
\end{array}
\Bigr]
a^k
\end{equation} 
where 
$\Bigl[ 
\begin{array}{c}
m \\
k
\end{array}
\Bigr]$ 
stands for the unsigned Stirling numbers of the first kind. We have
$$
\Bigl[ 
\begin{array}{c}
0 \\
0
\end{array}
\Bigr]=1, \hspace{1cm}
\Bigl[ 
\begin{array}{c}
m \\
0
\end{array}
\Bigr]=
\Bigl[ 
\begin{array}{c}
0 \\
m
\end{array}
\Bigr]=0, \hspace{1cm}
\Bigl[ 
\begin{array}{c}
m \\
m
\end{array}
\Bigr]=1.
$$
Moreover, the unsigned Stirling numbers of the first kind can be computed by the recurrence relation
which holds for all $m \geq 1$ and $k \geq 1$,
$$
\Bigl[ 
\begin{array}{c}
m+1 \\
k
\end{array}
\Bigr]
=m
\Bigl[ 
\begin{array}{c}
m \\
k
\end{array}
\Bigr]
+
\Bigl[ 
\begin{array}{c}
m \\
k-1
\end{array}
\Bigr].
$$
We deduce from \eqref{STIRLING} that for any integer $m \geq 1$
\begin{equation*}
\dE[ \Sigma_{n}^{(m)}]=\dE[\Sigma_n^m] + \sum_{k=1}^{m-1} 
\Bigl[ 
\begin{array}{c}
m \\
k
\end{array}
\Bigr]
\dE[\Sigma_n^k]
\end{equation*}
leading to
\begin{equation}
\label{NP9}
\dE[b_n^m \Sigma_{n}^{(m)}]=\dE[N_n^m] + \sum_{k=1}^{m-1} 
\Bigl[ 
\begin{array}{c}
m \\
k
\end{array}
\Bigr]
b_n^{m-k} \dE[N_n^k].
\end{equation}
Hence, it follows from the elementary fact that $b_n$ goes to zero, together with \eqref{NP6} and \eqref{NP9} as well as
standard results on the asymptotic behavior of the Euler Gamma function that
\begin{eqnarray}
\lim_{n \rightarrow \infty} \dE[N_n^m] &=& \lim_{n \rightarrow \infty} \dE[b_n^m \Sigma_{n}^{(m)}]= 
\lim_{n \rightarrow \infty} \Bigl(\frac{m!\Gamma(n+mb)}{\Gamma(n) \Gamma(1+mb)}\Bigr) \Bigl(\frac{\Gamma(n)\Gamma(b+1)}{\Gamma(n+b)}\Bigr)^m, \notag \\
&=&  \frac{m!(\Gamma(b+1))^m}{\Gamma(1+mb)} \lim_{n \rightarrow \infty} \Bigl(\frac{\Gamma(n+mb)}{\Gamma(n)}\Bigr) \Bigl(\frac{\Gamma(n)}{\Gamma(n+b)}\Bigr)^m, \notag \\
&=&  \frac{m!(\Gamma(b+1))^m}{\Gamma(1+mb)}
\label{NP10}
\end{eqnarray}
which is exactly what we wanted to prove.
\end{proof}

\noindent
{\bfseries Proof of Lemma \ref{L-ML}}. The proof of the almost sure convergence \eqref{ASCVGML} 
immediately follows from Lemma 
\ref{L-MARTN} together with the identity $N_n=b_n \Sigma_n$ and the elementary fact that 
$b=1-r$ as well as
\begin{equation}
\label{LIMBN}
\lim_{n \rightarrow \infty} n^b b_n=\lim_{n \rightarrow \infty} n^b 
\Big(\frac{\Gamma(n)\Gamma(b+1)}{\Gamma(n+b)}\Big)=\Gamma(b+1).
\end{equation}
Furthermore, it follows \eqref{MLMOM} that for all integer $m\geq 1$,
$$
\dE[\Sigma^m]=\frac{1}{( \Gamma(2-r))^m}\dE[N^m]=\frac{m!}{\Gamma(1+m(1-r))}.
$$
We recognize the moments of the Mittag-Leffler distribution given in \eqref{DEFMLMOM}.
Finally, as the Mittag-Leffler distribution is characterized by its moments, we can conclude that 
$\Sigma$ has a Mittag-Leffler distribution with parameter $1-r$,  which completes the proof of
Lemma \ref{L-ML}.
\demend
\vspace{-2ex} \\
\noindent
We are now in position where we can properly investigate the asymptotic behavior of the first martingale
$(M_n)$. It is closely related to that of the sequence $(v_n)$ defined, for all $n \geq 1$, by
$$
v_n=\sum_{k=1}^{n} \frac{a_{k}^2}{kb_k}.
$$
As a matter of fact, we already saw in Section \ref{S-I} that the location of the memory parameter
given by \eqref{MEMPAR} plays a crucial role. It is easy to see that
$$
p_r<\frac{3}{4} \Longleftrightarrow 2a< 1-r, 
\hspace{0.5cm}
p_r=\frac{3}{4} \Longleftrightarrow 2a= 1-r,
\hspace{0.5cm}
p_r>\frac{3}{4} \Longleftrightarrow 2a> 1-r.
$$
Moreover, the asymptotic behavior of $(v_n)$ in the three regimes is as follows. 
In the diffusive regime where $2a < 1-r$, that is $2a<b$, we have from \eqref{DEFABN} together with the well known 
asymptotic behavior of the Euler Gamma function that
\begin{equation}
\label{LIMVNDR}
\lim_{n \rightarrow \infty} \frac{v_n}{n^{1-r-2a}}= \lim_{n \rightarrow \infty} 
\frac{1}{n^{b-2a}} \sum_{k=1}^n \frac{1}{k} \Bigl(\frac{\Gamma(k)\Gamma(a+1)}{\Gamma(k+a)}\Bigr)^2
\Bigl(\frac{\Gamma(k+b)}{\Gamma(k)\Gamma(b+1)}\Bigr)
= \ell_r
\end{equation}
where
$$
\ell_r=\frac{1}{1-r-2a}\frac{(\Gamma(a+1))^2}{\Gamma(2-r)}.
$$
Hence, it follows from \eqref{ASCVGML} together with Toeplitz's lemma that
\begin{equation}
\label{CVGVNDR}
\lim_{n \rightarrow \infty} \frac{1}{n^{1-r-2a}}V_n
= \ell_r \Gamma(2-r)\Sigma  \hspace{1cm} \text{a.s.}
\end{equation}
In the critical regime where $2a=1-r$, that is $2a=b$, we obtain once again from \eqref{DEFABN} that
\begin{equation}
\label{LIMVNCR}
\lim_{n \rightarrow \infty} \frac{v_n}{\log n}= \lim_{n \rightarrow \infty} 
\frac{(\Gamma(a+1))^2}{\Gamma(b+1)} \frac{1}{\log n} \sum_{k=1}^n \frac{1}{k} = \frac{(\Gamma(a+1))^2}{\Gamma(2-r)} 
\end{equation}
which implies from \eqref{ASCVGML} and Toeplitz's lemma that
\begin{equation}
\label{CVGVNCR}
\lim_{n \rightarrow \infty} \frac{1}{\log n}V_n
= (\Gamma(a+1))^2 \Sigma  \hspace{1cm} \text{a.s.}
\end{equation}
In the superdiffusive regime where $2a>1-r$, 
$(v_n)$ converges to the finite value
\begin{equation} 
\label{LIMVNSR}
\lim_{n \rightarrow \infty}  v_n \!=\!
\sum_{k=0}^\infty \! \Big(\frac{\Gamma(a+1)\Gamma(k+1)}{\Gamma(k+a+1)} \Big)^2
\frac{\Gamma(k+b+1)}{\Gamma(k+2)\Gamma(b+1)}  \!=\! 
{}_{4}F_3 \Bigl( \begin{matrix}
{\hspace{0.1cm}1 , 1 , 1 ,2-r}\\
{2,a+1,a+1}\end{matrix} \Bigl|
{\displaystyle 1}\Bigr)
\end{equation} 
where $\!{}_{4}F_3$ stands for the hypergeometric function defined, for all $z \in \dC$, by
\begin{equation}
{}_{4}F_3 \Bigl( \begin{matrix}
{a,b,c,d}\\
{e,f,g}\end{matrix} \Bigl|
{\displaystyle z}\Bigr)
=\sum_{k=0}^{\infty}
\frac{a^{(k)}\, b^{(k)}\, c^{(k)} \,d^{(k)}}
{e^{(k)}\,f^{(k)}\,g^{(k)}\, k!} z^k.
\notag
\end{equation}
Therefore, we obtain from \eqref{ASCVGML} that $(V_n)$ converges 
almost surely to a finite and almost surely positive random variable $V$.
We shall see in the appendices that the above convergences will play a prominent role in order 
to investigate the asymptotic behavior of the ERWS.


\vspace{-1ex}
\section*{Appendix A \\ Proofs in the diffusive regime}
\renewcommand{\thesection}{\Alph{section}}
\renewcommand{\theequation}{\thesection.\arabic{equation}}
\setcounter{section}{1}
\setcounter{equation}{0}
\label{S-A}


\subsection{Almost sure convergence.}
We start with the proof of the almost sure convergence in the diffusive regime where $2a<1-r$.

\ \vspace{-1ex}\\
\noindent{\bfseries Proof of Theorem \ref{T-ASCVG-DR}.}
We obtain from the decomposition \eqref{CALCIPM} together with \eqref{CVGVNDR} that
$$
\langle M \rangle_n = O(n^{1-r-2a}) \hspace{1cm} \text{a.s.} 
$$
Then, it follows from the strong law of large numbers for martingales given e.g. by the last part of Theorem
1.3.24 in \cite{Duflo1997} that $M_n^2= O( \langle M \rangle_n \log \langle M \rangle_n )$ a.s. which immediately
implies that
\begin{equation}
\label{MNDR}
M_n^2=O(n^{1-r-2a} \log n) \hspace{1cm} \text{a.s.} 
\end{equation}
Consequently, as $M_n = a_nS_n$ and
\begin{equation}
\label{LIMAN}
\lim_{n \rightarrow \infty} n^a a_n=\lim_{n \rightarrow \infty} n^a 
\Big(\frac{\Gamma(n)\Gamma(a+1)}{\Gamma(n+a)}\Big)=\Gamma(a+1),
\end{equation}
we deduce from \eqref{MNDR} and \eqref{LIMAN} that
$S_n^2=O(n^{1-r} \log n)$ a.s. 
leading to
\begin{equation*}
\lim_{n \rightarrow \infty} \frac{S_n}{n}
= 0  \hspace{1cm} \text{a.s.}
\end{equation*}
One can also observe that we obtain from the identity $S_n^2=O(n^{1-r} \log n)$ that
\begin{equation}
\label{PRATEDR}
\lim_{n \rightarrow \infty} n^{r}\Big(\frac{S_n}{n}\Big)^2 = 0 \hspace{1cm}\text{a.s.}
\end{equation}
which will be useful in Section A.2.
\demend

\vspace{-2ex}
\subsection{Law of iterated logarithm.}
In order to prove the law of iterated logarithm in the diffusive regime, we shall first proceed to the calculation of $\dE[S_n^2]$ and $\dE[\langle M \rangle_n]$. We already saw from 
\eqref{NP6} with $m=1$ that for all $n \geq 1$,
\begin{equation}
\label{MOM1SIGMA}
\dE[\Sigma_n]=\frac{1}{b_n}.
\end{equation}
Moreover,
\eqref{CM2SUM} can be rewritten as
\begin{equation}
\dE[S_{n+1}^2|\cF_n]=\Big(1+ \frac{2a}{n} \Big)S_n^2
+b\frac{\Sigma_n}{n}  \hspace{1cm} \text{a.s.}
\label{PLILDR1}
\end{equation}
Hence, by taking the expectation on both sides of \eqref{PLILDR1}, we obtain that for all $n \geq 1$,
$$
\dE[S_{n+1}^2]=\Big(1+ \frac{2a}{n} \Big)\dE[S_n^2]+
+\frac{b}{nb_n}
$$
which leads to
\begin{eqnarray*}
\dE[S_n^2] & = &\frac{\Gamma(n+2a)}{\Gamma(n)\Gamma(2a +1)}\Big(1+ b 
\sum_{k=1}^{n-1} \frac{\Gamma(k+b)}{ \Gamma(k+1)\Gamma(b+1)}
\frac{\Gamma(k+1)\Gamma(2a +1)}{\Gamma(k+1+2a)}\Big), \\
& = &\frac{\Gamma(n+2a)}{\Gamma(n)\Gamma(2a +1)}\Big(1+ \frac{ \Gamma(2a +1)}{\Gamma(b)}
\sum_{k=1}^{n-1} \frac{\Gamma(k+b)}{\Gamma(k+1+2a)}\Big), \\
& = &
\frac{\Gamma(n+2a)}{\Gamma(n) \Gamma(b)}\sum_{k=1}^n \frac{\Gamma(k+b-1)}{\Gamma(k+2a)}.
\end{eqnarray*}
However, we deduce from Lemma B.1 in \cite{Bercu2018} that
$$
\sum_{k=1}^n \frac{\Gamma(k+b-1)}{\Gamma(k+2a)}=\frac{1}{b-2a}
\Big( \frac{\Gamma(n+b)}{\Gamma(n+2a)} - \frac{\Gamma(b)}{\Gamma(2a)}\Big).
$$
It clearly implies that for all $n \geq 1$,
\begin{equation}
\dE[S_{n}^2]=\frac{1}{b-2a} \Big(\frac{\Gamma(n+b)}{\Gamma(n) \Gamma(b)}- \frac{\Gamma(n+2a)}{\Gamma(n)\Gamma(2a)}\Big).
\label{MOM2SN}
\end{equation}
Hereafter, it follows from \eqref{CALCIPM} together with \eqref{MOM1SIGMA} and \eqref{MOM2SN}
\begin{eqnarray*}
\dE[\langle M \rangle_n] & = & 1 + b \sum_{k=1}^{n-1} \frac{a_{k+1}^2}{k} \dE[\Sigma_k]
-a^2\sum_{k=1}^{n-1} \frac{a_{k+1}^2}{k^2} \dE[S_k^2], \\
& = & b \sum_{k=1}^{n-1} \frac{a_{k+1}^2}{kb_k} +R_n
\end{eqnarray*}
where
$$
R_n = 1 +\frac{a^2}{b-2a}\Big(
\sum_{k=1}^{n-1} \frac{a_{k+1}^2\Gamma(k+2a)}{k^2\Gamma(k)\Gamma(2a)}
-b\sum_{k=1}^{n-1} \frac{a_{k+1}^2}{k^2b_k} \Big).
$$
We obtain from the well known asymptotic behavior of the Euler Gamma function that
$R_n=o(v_n)$ which ensures via \eqref{LIMVNDR} that
\begin{equation}
\label{CVGMEANIPMN}
\lim_{n \rightarrow \infty} \frac{1}{n^{1-r-2a}}\dE[\langle M \rangle_n]=(1-r) \ell_r.
\end{equation}

\ \vspace{-2ex}\\
\noindent{\bfseries Proof of Theorem \ref{T-LIL-DR}.}
We are now in position to prove the law of iterated logarithm for the martingale $(M_n)$
using Theorem 1 and Corollary 2 in \cite{Heyde1977}. First of all, we claim that
\begin{equation}
\label{CVGIPMNDR}
\lim_{n \rightarrow \infty} \frac{1}{n^{1-r-2a}}\langle M \rangle_n= (\Gamma(a+1))^2 \sigma_r^2 \Sigma
\hspace{1cm}\text{a.s.}
\end{equation}
where the asymptotic variance $\sigma_r^2$ is given by \eqref{VARDR}. As a matter of fact, we already saw
from \eqref{CVGVNDR} that
\begin{equation}
\label{PLILDR0}
\lim_{n \rightarrow \infty} \frac{1}{n^{1-r-2a}}V_n= \frac{(\Gamma(a+1))^2}{b} \sigma_r^2 \Sigma
\hspace{1cm}\text{a.s.}
\end{equation}
In addition, we obtain from \eqref{PRATEDR} together with Toeplitz's lemma that
\begin{equation}
\label{PLILDR2}
\lim_{n \rightarrow \infty} \frac{1}{n^{1-r-2a}}W_n
=\lim_{n \rightarrow \infty} \frac{(\Gamma(a+1))^2}{n^{1-r-2a}}
\sum_{k=1}^{n-1} \frac{1}{k^{r+2a}} k^{r}\Big(\frac{S_k}{k}\Big)^2
= 0  \hspace{1cm} \text{a.s.}
\end{equation}
Consequently, \eqref{CVGIPMNDR} follows from the conjunction of \eqref{CALCIPM}, 
\eqref{PLILDR0} and \eqref{PLILDR2}. Next, we are going to prove that
\begin{equation}
\label{}
\sum_{n=1}^\infty \frac{1}{n^{2(1-r-2a)}} \dE[|\Delta M_n|^4]  < \infty
\end{equation}
where $\Delta M_n=M_n-M_{n-1}=a_n \varepsilon_n$. Since $S_{n+1}=S_n+X_{n+1}$,
we get from \eqref{POSERWS} together with \eqref{CM12X} that
\begin{eqnarray}
\dE[S_{n+1}^3|\cF_n] &=&
\Big(1+\frac{3a}{n} \Big) S_n^3+\frac{3b}{n}S_n\Sigma_n +\frac{a}{n} S_n  \hspace{1cm} \text{a.s.} 
\label{CM3SUM} \\
\dE[S_{n+1}^4|\cF_n] &=&
\Big(1+\frac{4a}{n} \Big) S_n^4+\frac{6b}{n}S_n^2\Sigma_n +\frac{4a}{n} S_n^2 +b\frac{\Sigma_n}{n}  \hspace{1cm} \text{a.s.} 
\label{CM4SUM}
\end{eqnarray}
Hence, as $\varepsilon_{n+1}=S_{n+1} - \alpha_n S_n$, it follows from \eqref{CM1SUM}, \eqref{CM2SUM}, \eqref{CM3SUM} and \eqref{CM4SUM}
together with straightforward calculations that
\begin{eqnarray}
\dE[\varepsilon_{n+1}^3|\cF_n] &=&
 \frac{a S_n}{n}\Big(1 + 2\Big(\frac{aS_n}{n}\Big)^2 -\frac{3b\Sigma_n}{n}\Big) \hspace{1cm} \text{a.s.} 
\label{CM3EPS} \\
\dE[\varepsilon_{n+1}^4|\cF_n] &=& \frac{b \Sigma_n}{n}\Big(1 + \Big(\frac{aS_n}{n}\Big)^2\Big) -4\Big(\frac{aS_n}{n}\Big)^2
-3\Big(\frac{aS_n}{n}\Big)^4  \hspace{1cm} \text{a.s.} 
\label{CM4EPS}
\end{eqnarray}
Thus, we immediately deduce from \eqref{CM4EPS} that for all $n\geq 1$,
\begin{equation}
\label{MOMEPS4}
\dE[\varepsilon_{n+1}^4|\cF_n] \leq \frac{2b \Sigma_n}{n} \hspace{1cm} \text{a.s.}
\end{equation}
which, thanks to \eqref{MOM1SIGMA}, implies that for all $n\geq 1$,
\begin{equation*}
 \dE[\varepsilon^4_{n+1}]\leq \frac{2b}{nb_n}.
\end{equation*}
However, we obtain from Wendel's inequality for the ratio of two gamma functions that for all
$n \geq 1$,
$$
\frac{\Gamma(n+b)}{\Gamma(n)} \leq n^b
$$
leading, via \eqref{DEFABN}, to
\begin{equation}
\label{MOMEPS5}
 \dE[\varepsilon^4_{n+1}]\leq \frac{2}{\Gamma(b)n^{1-b}}.
\end{equation}
Therefore, as $b=1-r$, we obtain from \eqref{LIMAN} together with \eqref{MOMEPS5} that
\begin{equation}
\label{BSUMDR}
\sum_{n=1}^\infty \frac{1}{n^{2(1-r-2a)}} \dE[|\Delta M_n|^4] 
\leq 
1 +\frac{2}{\Gamma(b)}\sum_{n=1}^\infty \frac{a_{n}^4}{n^{b+1-4a}} <\infty.
\end{equation}
Furthermore, let $(P_n)$ be the martingale defined, for all $n \geq 1$, by
$$
P_n=\sum_{k=1}^n \frac{a_k^2}{k^{b-2a}} (\varepsilon_k^2 - \dE[\varepsilon_k^2 | \cF_{k-1}]).
$$
Its predictable quadratic variation is given by
$$
\langle P \rangle_n = \sum_{k=1}^n \frac{a_k^4}{k^{2(b-2a)}}
(\dE[\varepsilon_k^4| \mathcal{F}_{k-1}]- (\dE[\varepsilon_k^2 | \cF_{k-1}])^2).
$$
Hence, we obtain from \eqref{MOMEPS4} that
\begin{equation*}
\langle P \rangle_n \leq 2b\sum_{k=1}^n \frac{a_{k}^4}{k^{2b+1-4a}}\Sigma_k.
\end{equation*}
Consequently, we deduce from \eqref{ASCVGML} together with \eqref{LIMAN} that
$\langle P \rangle_n$ converges a.s. to a finite random variable. Then, it follows from
the strong law of large numbers for martingales given by the first part of
Theorem 1.3.15 in \cite{Duflo1997} that $(P_n)$ converges a.s. to a finite random variable.
Finally, all the conditions of Theorem 1 and Corollary 2 in \cite{Heyde1977} are satisfied, which leads to the
law of iterated logarithm
\begin{equation}
 \limsup_{n \rightarrow \infty} \frac{M_n}{\sqrt{2 \langle M \rangle_n  \log \log \langle M \rangle_n }}   = 
 -\liminf_{n \rightarrow \infty} \frac{M_n}{\sqrt{2 \langle M \rangle_n  \log \log \langle M \rangle_n }} 
 = 1\hspace{0.5cm} \text{a.s.}
 \label{PLILDR4}
\end{equation}
Therefore, as $M_n=a_n S_n$, we obtain from the almost sure convergence \eqref{ASCVGML} 
together with \eqref{LIMAN}, \eqref{CVGIPMNDR}  and \eqref{PLILDR4} that
\begin{equation}
 \limsup_{n \rightarrow \infty} \frac{S_n}{\sqrt{2 \Sigma_n \log \log \Sigma_n}}   = 
 -\liminf_{n \rightarrow \infty} \frac{S_n}{\sqrt{2 \Sigma_n \log \log \Sigma_n}} 
 = \sigma_r\hspace{1cm} \text{a.s.}
 \label{PLILDR5}
\end{equation}
We also deduce the law of iterated logarithm \eqref{LIL-DR3} from \eqref{ASCVGML} 
and \eqref{PLILDR5}, which completes the proof of
Theorem \ref{T-LIL-DR}. \demend

\vspace{-2ex}
\subsection{Asymptotic normality.}
\ \vspace{1ex}\\
\noindent{\bfseries Proof of Theorem \ref{T-AN-DR}.}
We shall now proceed to the proof of the asymptotic normality for the martingale $(M_n)$
using the first part of Theorem 1 and Corollaries 1 and 2 in \cite{Heyde1977}. We already saw that
\eqref{CVGIPMNDR} holds
and that $(P_n)$ converges almost surely to a finite random variable. It only remains to prove that
that for any $\eta>0$,
\begin{equation}
\lim_{n \rightarrow \infty}
\frac{1}{n^{1-r-2a}}\sum_{k=1}^{n}\dE\big[\Delta M_k^2 \rI_{\{|\Delta M_k|>\eta \sqrt{n^{1-r-2a}} \}}\big]=0.
\label{PANDR1}
\end{equation}
We clearly have for any $\eta>0$,
\begin{equation*}
\frac{1}{n^{1-r-2a}}\sum_{k=1}^{n}\dE\big[\Delta M_k^2 \rI_{\{|\Delta M_k|>\eta \sqrt{n^{1-r-2a}} \}}\big]
\leq 
\frac{1}{\eta^2 n^{2(1-r-2a)}}\sum_{k=1}^{n}\dE\big[\Delta M_k^4\big].
\end{equation*}
However, it was proven in \eqref{BSUMDR} that
\begin{equation*}
\sum_{n=1}^\infty \frac{1}{n^{2(1-r-2a)}} \dE[|\Delta M_n|^4] <\infty.
\end{equation*}
Then, it follows from Kronecker's lemma that
\begin{equation*}
\lim_{n \rightarrow \infty}
\frac{1}{n^{2(1-r-2a)}}\sum_{k=1}^{n}\dE\big[\Delta M_k^4 \big]=0
\end{equation*}
which immediately leads to \eqref{PANDR1}. Consequently, all the conditions of Theorem 1 and Corollaries 1 and 2 in \cite{Heyde1977}
are satisfied, which implies the asymptotic normality
\begin{equation}
\label{PANDR2}
\frac{M_n}{\sqrt{\langle M \rangle_n}} \liml \cN(0, 1).
\end{equation}
Moreover, we also deduce from Theorem 1 in \cite{Heyde1977} that
\begin{equation}
\label{PANDR3}
\frac{M_n}{\sqrt{n^{1-r-2a}}} \liml \Gamma(a+1) \sqrt{\Sigma^\prime}\cN\big(0, \sigma_r^2\big)
\end{equation}
where $\Sigma^\prime$ is independent of the Gaussian $\cN\big(0, \sigma_r^2\big)$ random variable
and $\Sigma^\prime$ shares the same distribution as $\Sigma$. Therefore,
as $M_n=a_n S_n$, we obtain from \eqref{LIMAN} that \eqref{PANDR3} reduces to
\begin{equation}
\label{PANDR4}
\frac{S_n}{\sqrt{n^{1-r}}} \liml \sqrt{\Sigma^\prime}\cN\big(0, \sigma_r^2\big).
\end{equation}
Finally, we find \eqref{ANDR} from \eqref{CVGIPMNDR}, \eqref{PANDR2} together with the almost sure convergence \eqref{ASCVGML} and Slutsky's lemma, 
which achieves the proof of Theorem
\ref{T-AN-DR}.
\demend


\vspace{-2ex}
\section*{Appendix B \\ Proofs in the critical regime}
\renewcommand{\thesection}{\Alph{section}}
\renewcommand{\theequation}{\thesection.\arabic{equation}}
\setcounter{section}{2}
\setcounter{equation}{0}
\setcounter{subsection}{0}
\label{S-B}


\subsection{Almost sure convergence.}
We carry on with the proof of the almost sure convergence in the critical regime where $2a=1-r$. 

\ \vspace{-1ex}\\
\noindent{\bfseries Proof of Theorem \ref{T-ASCVG-CR}.}
We obtain from \eqref{CALCIPM} and \eqref{CVGVNCR} that
$$
\langle M \rangle_n = O(\log n) \hspace{1cm} \text{a.s.} 
$$
Then, it follows from Theorem
1.3.24 in \cite{Duflo1997} that $M_n^2= O( \log n \log \log n )$ a.s. 
Consequently, as $M_n = a_nS_n$, we deduce from \eqref{LIMAN} that
$S_n^2=O(n^{1-r} \log n \log \log n)$ a.s. 
which clearly implies that
\begin{equation*}
\lim_{n \rightarrow \infty} \frac{S_n}{n}
= 0  \hspace{1cm} \text{a.s.}
\end{equation*}
As in the diffusive regime, we also obtain that
\begin{equation}
\label{PRATECR}
\lim_{n \rightarrow \infty} n^{r}\Big(\frac{S_n}{n}\Big)^2 = 0 \hspace{1cm}\text{a.s.}
\end{equation}
which will be useful in Section B.2.
\demend

\vspace{-2ex}
\subsection{Law of iterated logarithm.} 
\ \vspace{1ex}\\
\noindent{\bfseries Proof of Theorem \ref{T-LIL-CR}.}
The proof follows the same lines as that of Theorem \ref{T-LIL-DR}. We already saw from \eqref{CVGVNCR} that
\begin{equation}
\label{PLILCR1}
\lim_{n \rightarrow \infty} \frac{1}{\log n}V_n
= (\Gamma(a+1))^2 \Sigma  \hspace{1cm} \text{a.s.}
\end{equation}
Moreover, we have from \eqref{PRATECR} together with Toeplitz's lemma that
\begin{equation}
\label{PLILCR2}
\lim_{n \rightarrow \infty} \frac{1}{\log n}W_n
=\lim_{n \rightarrow \infty} \frac{(\Gamma(a+1))^2}{\log n}
\sum_{k=1}^{n-1} \frac{1}{k} k^{r}\Big(\frac{S_k}{k}\Big)^2
= 0  \hspace{1cm} \text{a.s.}
\end{equation}
Consequently, we obtain from \eqref{CALCIPM}, \eqref{PLILCR1} and \eqref{PLILCR2} that
\begin{equation}
\label{PLILCR3}
\lim_{n \rightarrow \infty} \frac{1}{\log n}\langle M \rangle_n
= (\Gamma(a+1))^2 b \Sigma  \hspace{1cm} \text{a.s.}
\end{equation}
In addition, as $2a=b$, we deduce from \eqref{LIMAN} and \eqref{BSUMDR} that
\begin{equation*}
\sum_{n=2}^\infty \frac{1}{(\log n)^2} \dE[|\Delta M_n|^4] 
\leq 
\frac{2}{\Gamma(b)}\sum_{n=1}^\infty \frac{1}{(\log n)^2} \frac{a_{n}^4}{n^{1-b}} <\infty.
\end{equation*}
Furthermore, let $(Q_n)$ be the martingale defined, for all $n \geq 1$, by
$$
Q_n=\sum_{k=2}^n \frac{a_k^2}{\log k} (\varepsilon_k^2 - \dE[\varepsilon_k^2 | \cF_{k-1}]).
$$
Its predictable quadratic variation satisfies
$$
\langle Q \rangle_n 
\leq 2b\sum_{k=1}^n \frac{a_{k}^4}{k(\log k)^2}\Sigma_k.
$$
Hence, we deduce from \eqref{ASCVGML} and \eqref{LIMAN} that
$\langle Q \rangle_n$ converges a.s. to a finite random variable which ensures that $(Q_n)$ 
converges a.s. to a finite random variable.
As in the diffusive regime, all the conditions of Theorem 1 and Corollary 2 in \cite{Heyde1977} are satisfied, which leads to the
law of iterated logarithm
\begin{equation}
 \limsup_{n \rightarrow \infty} \frac{M_n}{\sqrt{2 \langle M \rangle_n  \log \log \langle M \rangle_n }}   = 
 -\liminf_{n \rightarrow \infty} \frac{M_n}{\sqrt{2 \langle M \rangle_n  \log \log \langle M \rangle_n }} 
 = 1\hspace{0.5cm} \text{a.s.}
 \label{PLILCR4}
\end{equation}
Finally, as $M_n=a_n S_n$, the law of iterated logarithm \eqref{LIL-CR1} follows
from the almost sure convergence \eqref{ASCVGML} together with \eqref{LIMAN}, \eqref{PLILCR3} 
and \eqref{PLILCR4}.
We also obtain \eqref{LIL-CR3} from \eqref{ASCVGML} 
and \eqref{LIL-CR1}, which achieves the proof of
Theorem \ref{T-LIL-CR}. \demend

\vspace{-2ex}
\subsection{Asymptotic normality.}
\ \vspace{1ex}\\
\noindent{\bfseries Proof of Theorem \ref{T-AN-CR}.}
Via the same lines as in the proof of Theorem \ref{T-AN-DR}, we obtain the asymptotic normality
\begin{equation}
\label{PANCR1}
\frac{M_n}{\sqrt{\langle M \rangle_n}} \liml \cN(0, 1).
\end{equation}
In addition, we also deduce from Theorem 1 in \cite{Heyde1977} that
\begin{equation}
\label{PANCR2}
\frac{M_n}{\sqrt{ \log n}} \liml \Gamma(a+1) \sqrt{(1-r) \Sigma^\prime}\cN(0, 1)
\end{equation}
where $\Sigma^\prime$ is independent of the Gaussian $\cN(0,1)$ random variable
and $\Sigma^\prime$ shares the same distribution as $\Sigma$. Hence, we deduce \eqref{ANCR}
and \eqref{ANCRML} from \eqref{ASCVGML} , \eqref{PANCR1}, \eqref{PANCR2} and Slutsky's lemma, which completes the proof of Theorem
\ref{T-AN-CR}.
\demend


\vspace{-2ex}
\section*{Appendix C \\ Proofs in the superdiffusive regime}
\renewcommand{\thesection}{\Alph{section}}
\renewcommand{\theequation}{\thesection.\arabic{equation}}
\setcounter{section}{3}
\setcounter{equation}{0}
\setcounter{subsection}{0}
\label{S-C}

\noindent
In order to carry out the proofs in the superdiffusive regime where $2a>1-r$, it is necessary to show
that the martingale $(M_n)$ is bounded in $\dL^m$ for any integer $m\geq 1$. Denote by
$[M]_n$ the quadratic variation associated with $(M_n)$, given by
$[M]_0=0$ and, for all $n \geq 1$,
\begin{equation} 
\label{QVM}
[M]_n = \sum_{k=1}^n a_k^2\varepsilon_k^2.
\end{equation}
For all $n\geq 1$, the martingale increments are such that $\varepsilon_n^2 \leq 4$,
which implies that
$$
[M]_n \leq 4\sum_{k=1}^n a_k^2.
$$
Consequently, as soon as $2a>1$, we have for any integer $m \geq 1$,
$$
\sup_{n \geq 1} \dE\big[[M]_n^m\bigr] \leq 4^m \Big( 
\sum_{k=0}^{\infty}  \Bigl( \frac{\Gamma(k+1) \Gamma(a+1) }{\Gamma(k+a+1)} \Bigr)^2 \Big)^m
 <\infty.
$$
Therefore, it follows from the Burholder-Davis-Gundy inequality given e.g. by Theorem 2.10 in \cite{Hall1980} that,  for any real number $m>1$, there exists a positive constant $C_m$ such that
$$
\sup_{n \geq 1} \dE\big[|M_n|^m\bigr]  \leq C_m \sup_{n \geq 1} \dE\big[[M]_n^{m/2}\bigr] 
 <\infty.
$$
It is much more difficult to show this result under the only hypothesis $2a>1-r$.

\begin{lem}
\label{L-MARTM}
In the superdiffusive regime, we have for any integer $m\geq 1$,
\begin{equation}
\label{SUPQVM}
\sup_{n \geq 1} \dE\big[[M]_n^m\bigr]<\infty.
\end{equation}
Consequently, the martingale $(M_n)$ converges almost surely and in $\dL^m$ to a finite random variable M.
\end{lem}

\begin{proof}
We shall prove \eqref{SUPQVM} by induction on $m \geq 1 $ and by the calculation of
$$
\dE[[M]_n^{(m)}]=\dE[[M]_n([M]_n+1)\cdots([M]_n+m-1)].
$$
For $m=1$, we have from \eqref{CM2EPS}, \eqref{MOM1SIGMA} and \eqref{QVM} that
\begin{equation}
\label{MP1}
\dE[[M]_n]\leq  1+b \sum_{k=1}^{n-1} \frac{a_{k}^2}{k}\dE[\Sigma_{k}]
\leq 1+b \sum_{k=1}^{n-1} \frac{a_{k}^2}{kb_k}
\leq 1+b v_{n}.
\end{equation}
Then, we immediately deduce from \eqref{LIMVNSR} and \eqref{MP1} that
\begin{equation*}
\sup_{n \geq 1} \dE\big[[M]_n\bigr]<\infty.
\end{equation*}
We also claim that 
\begin{equation}
\label{MP2}
\sup_{n \geq 1} \dE\big[N_n [M]_n\big]<\infty.
\end{equation}
As a matter of fact, it follows from \eqref{CM12X} and \eqref{CM2EPS} that for all $n\geq 1$,
\begin{equation}
\dE[\Sigma_{n+1} [M]_{n+1} |\cF_n] \leq \Big(1+ \frac{b}{n} \Big)\Sigma_n [M]_n
+\frac{b a_{n+1}^2}{n} \Sigma_n + \frac{b a_{n+1}^2}{n}\Sigma_n^2  \hspace{1cm} \text{a.s.}
\label{MP3}
\end{equation}
Hence, by taking the expectation on both sides of \eqref{MP3}, we obtain that for all $n \geq 1$,
$$
\dE[\Sigma_{n+1} [M]_{n+1}]\leq \Big(1+ \frac{b}{n} \Big)\dE[\Sigma_n [M]_n]
+\frac{b a_{n+1}^2 }{n} \dE[\Sigma_n] + \frac{b a_{n+1}^2}{n}\dE[\Sigma_n^2].
$$
However, we find from \eqref{NP6} with $m=2$ that $\dE[\Sigma_n^2]=2b_n(2) -\dE[\Sigma_n]$ where 
\begin{equation*}
b_n(2)=\frac{\Gamma(n+2b)}{\Gamma(n) \Gamma(1+2b)}.
\end{equation*}
It ensures that
\begin{eqnarray*}
\dE[\Sigma_n [M]_n] & \leq &\frac{\Gamma(n+b)}{\Gamma(n)\Gamma(b +1)}\Big(1+ 2b 
\sum_{k=1}^{n-1} a_{k+1}^2 b_{k+1}  \frac{b_k(2)}{k}\Big), \\
 & \leq &\frac{\Gamma(n+b)}{\Gamma(n)\Gamma(b +1)}\Big(1+ 2b 
\sum_{k=1}^{n-1} a_{k+1}^2 b_{k+1}  \frac{b_{k+1}(2)}{ k+2b}\Big), \\
& \leq &\frac{\Gamma(n+b)}{\Gamma(n)\Gamma(b +1)}\Big(2b\sum_{k=1}^n a_{k}^2 b_{k} \frac{b_{k}(2)}{k+2b-1}\Big), \\
& \leq &\frac{2\Gamma(n+b)}{\Gamma(n)\Gamma(b +1)}\sum_{k=1}^n a_{k}^2 b_{k} \frac{b_{k}(2)}{k}, \\
& \leq &\frac{2\Gamma(n+b)}{\Gamma(n)\Gamma(b +1)}\sum_{k=1}^n \frac{a_{k}^2}{k b_k}, 
\end{eqnarray*}
thanks to \eqref{NP7} with $m=2$. Consequently, we deduce from \eqref{LIMVNSR} that there exists a constant $c_1>0$ such that for all
$n \geq 1$,
\begin{equation}
\label{INDMS}
\dE[\Sigma_n [M]_n] \leq \frac{c_1}{b_n}
\end{equation}
which clearly leads to \eqref{MP2} as $N_n=b_n \Sigma_n$.
From now on, assume that $m \geq 2$ and that for all $0 \leq k \leq m-1$, there exists a constant $c_k>0$ such that for all
$n \geq 1$,
\begin{equation}
\label{INDM}
\dE[\Sigma_n [M]_n^{(k)}] \leq \frac{c_k}{b_n}.
\end{equation}
It follows from \eqref{QVM} that 
\begin{equation}
\label{MP4}
[M]_{n+1}^{(m)}=([M]_{n}+a_{n+1}^2\varepsilon_{n+1}^2)^{(m)}=\sum_{k=0}^m
\begin{pmatrix}
m \\ k
\end{pmatrix}
[M]_n^{(k)}(a_{n+1}^2\varepsilon_{n+1}^2)^{(m-k)}.
\end{equation}
By taking the conditional expectation on both sides of \eqref{MP4}, we obtain that
\begin{equation}
\label{MP5}
\dE[[M]_{n+1}^{(m)}| \cF_n]=\sum_{k=0}^m
\begin{pmatrix}
m \\ k
\end{pmatrix}
[M]_n^{(k)}\dE[ (a_{n+1}^2\varepsilon_{n+1}^2)^{(m-k)}| \cF_n].
\end{equation}
However, as $a_n^2 \leq 1$ and $\varepsilon_n^2 \leq 4$, we find from \eqref{CM2EPS} that
for all $k\geq 1$,
\begin{equation}
\label{MP6}
\dE[ (a_{n+1}^2\varepsilon_{n+1}^2)^{(k)}| \cF_n] \leq 4^{(k)} a_{n+1}^2 \dE[\varepsilon_{n+1}^2| \cF_n]
\leq \frac{4^{(k)} b a_{n+1}^2}{n}\Sigma_n \hspace{1cm} \text{a.s.}
\end{equation}
Consequently, we deduce from \eqref{MP5} and \eqref{MP6} that 
\begin{equation*}
\dE[[M]_{n+1}^{(m)}| \cF_n]\leq [M]_n^{(m)}+\frac{b a_{n+1}^2}{n}
\sum_{k=0}^{m-1}
\begin{pmatrix}
m \\ k
\end{pmatrix} 4^{(m-k)}   \Sigma_n [M]_n^{(k)}\hspace{1cm} \text{a.s.} 
\end{equation*}
which leads via \eqref{INDM} to
\begin{eqnarray}
\dE[[M]_{n+1}^{(m)}]| &\leq& \dE[[M]_n^{(m)}]+\frac{b a_{n+1}^2}{n} \sum_{k=0}^{m-1}
\begin{pmatrix}
m \\ k
\end{pmatrix} 4^{(m-k)} \dE[\Sigma_n[M]_n^{(k)}], \notag\\
&\leq& \dE[[M]_n^{(m)}]+\frac{b a_{n+1}^2}{nb_n} \sum_{k=0}^{m-1}
\begin{pmatrix}
m \\ k
\end{pmatrix} 4^{(m-k)} c_k.
\label{MP7}
\end{eqnarray}
Hereafter, denote
$$
d_m=\sum_{k=0}^{m-1}
\begin{pmatrix}
m \\ k
\end{pmatrix} 4^{(m-k)} c_k.
$$
We clearly have from \eqref{MP7} that 
\begin{equation}
\label{MP8}
\dE[[M]_n^{(m)}]\leq  1+b d_m\sum_{k=1}^{n-1}\frac{a_{k}^2}{kb_k}
\leq 1+b d_mv_{n}.
\end{equation}
Hence, we obtain from \eqref{LIMVNSR} and \eqref{MP8} that
\begin{equation}
\label{MP9}
\sup_{n \geq 1} \dE\big[[M]_n^{(m)}\bigr]<\infty.
\end{equation}
The proof that there exists a constant $c_m>0$ such that for all
$n \geq 1$,
\begin{equation*}
\dE[\Sigma_n [M]_n^{(m)}] \leq \frac{c_m}{b_n}
\end{equation*}
is left to the reader inasmuch as it follows essentially the same lines as that of \eqref{INDMS}.
Consequently, we also find that
\begin{equation*}
\sup_{n \geq 1} \dE\big[N_n[M]_n^{(m)}\bigr]<\infty.
\end{equation*}
We immediately deduce from \eqref{MP9} that for any integer $m\geq 1$,
\begin{equation}
\label{MP10}
\sup_{n \geq 1} \dE\big[[M]_n^m\bigr] \leq \sup_{n \geq 1} \dE\big[[M]_n^{(m)}\bigr] <\infty.
\end{equation}
Finally, we obtain from \eqref{MP10} together with the Burholder-Davis-Gundy inequality that
the martingale $(M_n)$ is bounded in $\dL^m$ for any integer $m\geq 1$. We can conclude that $(M_n)$ converges almost surely and
in $\dL^m$ to a finite random variable $M$, which achieves the proof of Lemma \ref{L-MARTM}. 
\end{proof}

\ \vspace{-1ex}\\
\noindent{\bfseries Proof of Theorem \ref{T-ASCVG-SR}.}
We are now in position to prove the almost sure convergence \eqref{ASCVGSR1}. It was just shown in
Lemma \ref{L-MARTM} that the martingale $(M_n)$ converges almost surely to a finite random variable 
$M$. Consequently, as $M_n=a_n S_n$,
we immediately deduce from \eqref{LIMAN} with $a=2p+r-1$, that
\begin{equation*}
\lim_{n \rightarrow \infty} \frac{S_n}{n^{2p+r-1}}
= L  \hspace{1cm} \text{a.s.}
\end{equation*}
where the limiting random variable $L$ is given by
\begin{equation}
\label{DEFL}
L=\frac{1}{\Gamma(2p+r)}M.
\end{equation}
Moreover, it also follows from Lemma \ref{L-MARTM} that for any integer $m\geq 1$,
\begin{equation}
\label{PLMCVGMSR}
\lim_{n \rightarrow \infty} \dE[ |M_n -M|^m]
= 0.
\end{equation}
Dividing both sides of \eqref{PLMCVGMSR} by $\Gamma^m(2p+r)$, we obtain from \eqref{LIMAN} 
and \eqref{DEFL} that for any integer $m\geq 1$,
\begin{equation*}
\label{LMCVG}
\lim_{n \rightarrow \infty} \dE\Bigl[ \Bigl|\frac{S_n}{n^{2p+r-1}} -L\Bigr|^m\Bigr]
= 0
\end{equation*}
which is exactly what we wanted to prove.
\demend

\ \vspace{-1ex}\\
\noindent{\bfseries Proof of Theorem \ref{T-MOM-SR}.}
Hereafter, we are going to compute the first four moments of the random variable $L$ where we recall
that $a=2p+r-1$ and $b=1-r$.
We have from \eqref{CM1SUM} that for all $n \geq 1$,
$$
\dE[S_{n+1}]=\Big(1+ \frac{a}{n} \Big)\dE[S_n]
$$
which leads to
\begin{equation}
\label{PML1}
\dE[S_n]=\prod_{k=1}^{n-1} \Big(1+ \frac{a}{k} \Big)\dE[S_1] =\frac{2s-1}{a_n}.
\end{equation}
Hence, we immediately get from \eqref{PML1} that
$$
\dE[L]  = \lim_{n \rightarrow \infty} \frac{\dE[a_nS_n]}{\Gamma(a+1)}  =
\frac{2s-1}{a\Gamma(a)}
=\frac{2s-1}{(2p+r-1)\Gamma(2p+r-1)}.
$$
In addition, it follows from \eqref{MOM2SN} that
$$
\dE[L^2]  = \lim_{n \rightarrow \infty} \frac{\dE[a_n^2S_n^2]}{\Gamma^2(a+1)}  =
\frac{1}{(2a-b)\Gamma(2a)}=\frac{1}{(4p+3(r-1))\Gamma(2(2p+r-1))}.
$$
Moreover, we obtain from \eqref{CM3SUM} that for all $n \geq 1$,
\begin{equation}
\label{PML2}
\dE[S_{n+1}^3] =
\Big(1+\frac{3a}{n} \Big) \dE[S_n^3]+\frac{3b}{n}\dE[S_n\Sigma_n] +\frac{a}{n} \dE[S_n].
\end{equation}
However, one can easily check that for all $n \geq 1$,
\begin{eqnarray}
\dE[S_n\Sigma_n] & = & \frac{(2s-1)\Gamma(n+a+b)}{\Gamma(n)\Gamma(a+b+1)}\Big(1+
\frac{\Gamma(a+b+1)}{\Gamma(a)}\sum_{k=1}^{n-1} \frac{\Gamma(k+a)}{\Gamma(k+1+a+b)}\Big), \notag \\
& = &
\frac{(2s-1)\Gamma(n+a+b)}{\Gamma(n)\Gamma(a)}
\sum_{k=1}^n \frac{\Gamma(k+a-1)}{\Gamma(k+a+b)}, \notag \\
&=& \frac{(2s-1)\Gamma(n+a+b)}{b\Gamma(n)\Gamma(a)}
\Big( \frac{\Gamma(a)}{\Gamma(a+b)} -\frac{\Gamma(n+a)}{\Gamma(n+a+b)}\Big), \notag \\
&=& \frac{(2s-1)}{b}\Big(\frac{\Gamma(n+a+b)}{\Gamma(n)\Gamma(a+b)} - \frac{a}{a_n} \Big).
\label{PML3}
\end{eqnarray}
Hence, we deduce from \eqref{PML1}, \eqref{PML2} and \eqref{PML3}  that for all $n \geq 1$,
\begin{equation*}
\dE[S_{n+1}^3] =
\Big(1+\frac{3a}{n} \Big) \dE[S_n^3]+(2s-1) \Big(
\frac{3\Gamma(n+a+b)}{\Gamma(n+1)\Gamma(a+b)} -\frac{2\Gamma(n+a)}{\Gamma(n+1)\Gamma(a)} \Big)
\end{equation*}
which leads to
\begin{equation}
\label{PML4}
\dE[S_n^3]  = \frac{(2s-1)\Gamma(n+3a)}{\Gamma(n)\Gamma(3a+1)}\Big(1+
\frac{3\Gamma(3a+1)}{\Gamma(a+b)}\xi_n
-\frac{2\Gamma(3a+1)}{\Gamma(a)}\zeta_n\Big)
\end{equation}
where
\begin{eqnarray*}
\xi_n&=&\sum_{k=1}^{n-1} \frac{\Gamma(k+a+b)}{\Gamma(k+1+3a)} =
\sum_{k=1}^n \frac{\Gamma(k+a+b-1)}{\Gamma(k+3a)} -\frac{\Gamma(a+b)}{\Gamma(3a+1)},\\
\zeta_n&=&\sum_{k=1}^{n-1} \frac{\Gamma(k+a)}{\Gamma(k+1+3a)} =
\sum_{k=1}^n \frac{\Gamma(k+a-1)}{\Gamma(k+3a)} -\frac{\Gamma(a)}{\Gamma(3a+1)}.
\end{eqnarray*}
However, we find from Lemma B.1 in \cite{Bercu2018} that
$$
\sum_{k=1}^n \frac{\Gamma(k+a+b-1)}{\Gamma(k+3a)}=\frac{1}{2a-b}
\Big( \frac{\Gamma(a+b)}{\Gamma(3a)} -\frac{\Gamma(n+a+b)}{\Gamma(n+3a)}\Big),
$$
$$
\sum_{k=1}^n \frac{\Gamma(k+a-1)}{\Gamma(k+3a)}=\frac{1}{2a}
\Big( \frac{\Gamma(a)}{\Gamma(3a)} -\frac{\Gamma(n+a)}{\Gamma(n+3a)}\Big).
$$
Consequently, we obtain from \eqref{PML3} that for all $n \geq 1$,
\begin{equation}
\label{PML5}
\dE[S_n^3]  = \frac{(2s-1)}{\Gamma(n)}\Big(
\frac{(a+b)\Gamma(n+3a)}{a(2a-b)\Gamma(3a)}-
\frac{3\Gamma(n+a+b)}{(2a-b)\Gamma(a+b)}
+\frac{\Gamma(n+a)}{\Gamma(a+1)}\Big).
\end{equation}
Therefore, it follows from \eqref{PML5} that
\begin{eqnarray*}
\dE[L^3]  &=& \lim_{n \rightarrow \infty} \frac{\dE[a_n^3S_n^3]}{\Gamma^3(a+1)}  =
\frac{(2s-1)(a+b)}{a(2a-b)\Gamma(3a)}, \\
&=& \frac{2p(2s-1)}{(2p+r-1)(4p+3(r-1))\Gamma(3(2p+r-1))}.
\end{eqnarray*}
By the same token, we deduce from \eqref{CM4SUM} that for all $n \geq 1$,
\begin{equation}
\label{PML6}
\dE[S_{n+1}^4] =
\Big(1+\frac{4a}{n} \Big) \dE[S_n^4]+\frac{6b}{n}\dE[S_n^2\Sigma_n] +\frac{4a}{n} \dE[S_n^2] +\frac{b}{n}
\dE[\Sigma_n].
\end{equation}
Via the same lines as in the proof of \eqref{PML3}, we obtain that for all $n\geq 1$,
\begin{equation}
\label{PML7}
\dE[S_n^2\Sigma_n]=\frac{2a\Gamma(n+c)}{b(2a-b)\Gamma(n)\Gamma(c)}
-\frac{1}{(2a-b)}\Big(\frac{2a\Gamma(n+2a)}{b\Gamma(n)\Gamma(2a)}+
\frac{\Gamma(n+2b)}{\Gamma(n)\Gamma(2b)}-\frac{b}{b_n}\Big)
\end{equation}
where $c=2a+b$. Hence, it follows from \eqref{MOM2SN}, \eqref{PML6} and \eqref{PML7}
together with tedious but straighforward calculations that for all $n\geq 1$,
\begin{equation}
\label{PML8}
\dE[S_n^4]  = \frac{\Gamma(n+4a)}{(2a-b)\Gamma(n)}\Big(
\frac{12aP_n}{\Gamma(2a+b)} -\frac{8aQ_n}{\Gamma(2a)}
-\frac{6bR_n}{\Gamma(2b)} +\frac{(5b-2a)T_n}{\Gamma(b)}\Big)
\end{equation}
where
\begin{align*}
P_n&=\sum_{k=1}^n \frac{\Gamma(k+c-1)}{\Gamma(k+4a)}=\frac{1}{2a-b}\Big(\frac{\Gamma(2a+b)}{\Gamma(4a)}-
\frac{\Gamma(n+2a+b)}{\Gamma(n+4a)}\Big), \\
Q_n&=\sum_{k=1}^n \frac{\Gamma(k+2a-1)}{\Gamma(k+4a)}=\frac{1}{2a}\Big(\frac{\Gamma(2a)}{\Gamma(4a)}-
\frac{\Gamma(n+2a)}{\Gamma(n+4a)}\Big), \\
R_n&=\sum_{k=1}^n \frac{\Gamma(k+2b-1)}{\Gamma(k+4a)}=\frac{1}{2(2a-b)}\Big(\frac{\Gamma(2b)}{\Gamma(4a)}-
\frac{\Gamma(n+2b)}{\Gamma(n+4a)}\Big), \\
T_n&=\sum_{k=1}^n \frac{\Gamma(k+b-1)}{\Gamma(k+4a)}=\frac{1}{4a-b}\Big(\frac{\Gamma(b)}{\Gamma(4a)}-
\frac{\Gamma(n+b)}{\Gamma(n+4a)}\Big).
\end{align*}
Finally, we find from \eqref{PML8} that
\begin{eqnarray*}
\dE[L^4]  &=& \lim_{n \rightarrow \infty} \frac{\dE[a_n^4S_n^4]}{\Gamma^4(a+1)}  =
\frac{3\big((4a-b)^2-3(2a-b)^2\big)}{(4a-b) (2a-b)^2\Gamma(4a)}, \\
&=& \frac{6(2a^2+2ab-b^2)}{(4a-b)(2a-b)^2\Gamma(4a)}, \\
&=& \frac{6\big(8p^2-4p(1-r)-(1-r)^2\big)}{(8p+5(r-1))(4p+3(r-1))^2\Gamma(4(2p+r-1))}
\end{eqnarray*}
which completes the proof of Theorem \ref{T-MOM-SR}.
\demend

\vspace{-2ex}
\subsection{Asymptotic normality.}
\ \vspace{1ex}\\
\noindent{\bfseries Proof of Theorem \ref{T-AN-SR}.}
We are going to prove that the fluctuation of the ERWS around its limiting random variable $L$ is still Gaussian.
In contrast with the diffusive and critical regimes, the proof of the asymptotic normality for the martingale $(M_n)$
in the superdiffusive regime relies on the second part of Theorem 1 and Corollaries 1 and 2 in \cite{Heyde1977}.
Denote
$$
\Lambda_n=\sum_{k=n}^{\infty}\dE\big[\Delta M_k^2|\cF_{k-1}].
$$
It follows from \eqref{CM2EPS} that for all $n\geq 2$,
\begin{equation}
\label{PANSR1}
\Lambda_n=\sum_{k=n}^{\infty}a_k^2\dE\big[\varepsilon^2_k|\cF_{k-1}]
=b\sum_{k=n}^{\infty}a_{k}^2 \frac{\Sigma_{k-1}}{k\!-\!1}-a^2 \sum_{k=n}^{\infty} a_k^2\Big(\frac{S_{k-1}}{k\!-\!1}\Big)^2.
\end{equation}
On the one hand, we have from the almost sure convergence \eqref{ASCVGML} that
\begin{equation}
\label{PANSR2}
\lim_{n \rightarrow \infty} n^{2a+r-1}\sum_{k=n}^{\infty}a_{k}^2 \frac{\Sigma_{k-1}}{k\!-\!1}= \frac{(\Gamma(a+1))^2}{2a+r-1}  \Sigma
\hspace{1cm}\text{a.s.}
\end{equation}
On the other hand, we also obtain from the almost sure convergence \eqref{ASCVGSR1} that
\begin{equation}
\label{PANSR3}
\lim_{n \rightarrow \infty} n\sum_{k=n}^{\infty}a_{k}^2 \Big(\frac{S_{k-1}}{k\!-\!1}\Big)^2= (\Gamma(a+1))^2  L^2
\hspace{1cm}\text{a.s.}
\end{equation}
One can observe that we always have $2a+r-1<1$, which means that the second term in \eqref{PANSR1}
plays a negligible role. Consequently, we deduce from \eqref{PANSR1}, \eqref{PANSR2} and \eqref{PANSR3} that
\begin{equation}
\label{PANSR4}
\lim_{n \rightarrow \infty} n^{2a+r-1}\Lambda_n= (\Gamma(a+1))^2\tau_r^2  \Sigma
\hspace{1cm}\text{a.s.}
\end{equation}
where the asymptotic variance $\tau_r^2$ is given by \eqref{VARSR}. 
Moreover, we find from \eqref{MOMEPS5} that for any $\eta>0$,
\begin{eqnarray}
n^{2a+r-1}\sum_{k=n}^{\infty}\dE\big[\Delta M_k^2 \rI_{\{|\Delta M_k|>\eta \sqrt{n^{1-r-2a}} \}}\big]
&\leq &
\frac{1}{\eta^2}n^{2(2a+r-1)}\sum_{k=n}^{\infty}\dE\big[\Delta M_k^4\big], \notag\\
&\leq &
\frac{2}{\eta^2\Gamma(b)} n^{2(2a+r-1)}\sum_{k=n}^{\infty} \frac{a_k^4}{k^{r}}.
\label{PANSR5}
\end{eqnarray}
However, one can easily see that
\begin{equation*}
\lim_{n \rightarrow \infty} n^{4a+r-1}\sum_{k=n}^{\infty} \frac{a_k^4}{k^{r}}=\frac{(\Gamma(a+1))^4}{4a+r-1}.
\end{equation*}
Hence, we obtain from the upper bound in \eqref{PANSR5} that for any $\eta>0$,
\begin{equation*}
\lim_{n \rightarrow \infty} n^{2a+r-1}\sum_{k=n}^{\infty}\dE\big[\Delta M_k^2 \rI_{\{|\Delta M_k|>\eta \sqrt{n^{1-r-2a}} \}}\big]=0.
\end{equation*}
Furthermore, let $(P_n)$ be the martingale defined, for all $n \geq 1$, by
$$
P_n=\sum_{k=1}^n k^{2a-b}a_k^2 (\varepsilon_k^2 - \dE[\varepsilon_k^2 | \cF_{k-1}]).
$$
Its predictable quadratic variation is given by
$$
\langle P \rangle_n = \sum_{k=1}^n k^{2(2a-b)} a_k^4
(\dE[\varepsilon_k^4| \mathcal{F}_{k-1}]- (\dE[\varepsilon_k^2 | \cF_{k-1}])^2).
$$
Therefore, we obtain from \eqref{MOMEPS4} that
\begin{equation*}
\langle P \rangle_n \leq 2b\sum_{k=1}^n k^{4a-1-2b}a_{k}^4\Sigma_k
\end{equation*}
which implies via the almost sure convergence \eqref{ASCVGML} and \eqref{LIMAN} that
$\langle P \rangle_n$ converges a.s. to a finite random variable. Then, we deduce once again from
the strong law of large numbers for martingales that $(P_n)$ converges a.s. to a finite random variable.
Finally, all the conditions of the second part of Theorem 1 and Corollaries 1 and 2 in \cite{Heyde1977} are satisfied, which leads
to the asymptotic normality
\begin{equation}
\label{PANSR6}
\frac{M_n-M}{\sqrt{\Lambda_n}} \liml \cN(0, 1).
\end{equation}
Moreover, we also deduce from Theorem 1 in \cite{Heyde1977} that
\begin{equation}
\label{PANSR7}
\sqrt{n^{2a+r-1}}(M_n-M) \liml \Gamma(a+1) \sqrt{\Sigma^\prime}\cN\big(0, \tau_r^2\big)
\end{equation}
where $\Sigma^\prime$ is independent of the Gaussian $\cN\big(0, \tau_r^2\big)$ random variable
and $\Sigma^\prime$ shares the same distribution as $\Sigma$. 
Furthermore, we recall from \eqref{DEFL} that $M$ and $L$ are tightly related by the identity
$M=\Gamma(a+1)$.
Therefore,
as $M_n=a_n S_n$, we obtain from \eqref{LIMAN} that \eqref{PANSR7} reduces to
\begin{equation*}
\label{PANSR8}
\frac{S_n -n^{a}L }{\sqrt{n^{1-r}}} 
\liml \sqrt{\Sigma^\prime}\cN\big(0,\tau_r^2\big),
\end{equation*}
which coincides with \eqref{ANSR2} as $a=2p+r-1$.
Finally, we find \eqref{ANSR1} from \eqref{PANSR4},  \eqref{PANSR6} together with the almost sure convergence \eqref{ASCVGML} 
and Slutsky's lemma, which 
completes the proof of Theorem
\ref{T-AN-SR}.
\demend

\bibliographystyle{abbrv}
\bibliography{Biblio-ERWSTOP}

\end{document}